\documentclass[11pt]{amsart}
\usepackage{amsmath,amssymb,amsthm}
\usepackage{tikz}
\usetikzlibrary{arrows.meta,decorations.pathreplacing}
\usepackage[colorlinks=true,linkcolor=blue,citecolor=blue,urlcolor=blue]{hyperref}

\newtheorem{theorem}{Theorem}[section]

\newtheorem{proposition}[theorem]{Proposition}
\newtheorem{corollary}[theorem]{Corollary}
\theoremstyle{definition}
\newtheorem{definition}[theorem]{Definition}
\newtheorem{remark}[theorem]{Remark}
\newtheorem{example}[theorem]{Example}
\newtheorem{question}[theorem]{Question}
\newtheorem{conjecture}[theorem]{Conjecture}

\newcommand{\B}{B}
\newcommand{\N}{\mathbb N}
\newcommand{\R}{\mathbb R}
\newcommand{\pl}{\operatorname{pl}}
\newcommand{\CanPlat}{\operatorname{CanPlat}}
\newcommand{\minD}{\operatorname{min}_{<_D}}

\title{A Dehornoy-Type Ordering on Bridge Positions of Links}
\author{Makoto Ozawa}
\address{Department of Natural Sciences, Faculty of Arts and Sciences, Komazawa University, 1-23-1 Komazawa, Setagaya-ku, Tokyo, 154-8525, Japan}
\email{w3c@komazawa-u.ac.jp}
\subjclass[2020]{Primary 57K10; Secondary 20F36, 20F60}
\keywords{bridge position, bridge isotopy, plat presentation, Hilden subgroup, double coset, braid group, Dehornoy order}

\begin{document}

\begin{abstract}
By a theorem of Hovland, bridge isotopy classes of $n$--bridge positions of links in $S^3$ correspond bijectively to Hilden double cosets
\[
\mathcal P_{2n}=H_{2n}\backslash \B_{2n}/H_{2n}.
\]
We use this correspondence to introduce Dehornoy-type orderings on bridge positions. After fixing a proper complexity function on $\B_{2n}$, we assign to each double coset its Dehornoy-minimal braid among the elements of minimal complexity. Comparing these representatives by the Dehornoy order gives a strict total order $\prec_{D,n}$, and comparing them first by complexity and then by the Dehornoy order gives a strict well-order $\mathrel{\lhd_{D,n}}$. For a link type $\mathcal L$ and $n\ge b(\mathcal L)$, we single out a canonical $n$--bridge position $\CanPlat_D^{(n)}(\mathcal L)$: it is the $\mathrel{\lhd_{D,n}}$--least element of the set of $n$--bridge positions of $\mathcal L$, and the $\prec_{D,n}$--least element among those of minimal complexity, and its canonical braid coincides with the globally Dehornoy-minimal braid of minimal complexity representing $\mathcal L$. We show by an example that, in general, it need not be least for $\prec_{D,n}$ on the whole set. Finally, we show that the fixed-level bridge finiteness conjecture is equivalent to boundedness of the complexities of the canonical representatives, and to the existence of a $\mathrel{\lhd_{D,n}}$--greatest $n$--bridge position.
\end{abstract}

\maketitle

\section{Introduction}

A link $L\subset S^3$ is in \emph{$n$--bridge position} with respect to a standard height function if it has exactly $n$ local maxima and $n$ local minima and every maximum lies above every minimum. Two such links are \emph{bridge isotopic} if they are connected by an isotopy through bridge positions. Every link in $n$--bridge position is, after straightening its maxima and minima, the plat closure of a braid on $2n$ strands. Hovland \cite{Hovland} proved that this construction identifies the set of bridge isotopy classes of $n$--bridge positions with the Hilden double coset space
\[
\mathcal P_{2n}=H_{2n}\backslash \B_{2n}/H_{2n},
\]
where $H_{2n}\le \B_{2n}$ is the Hilden subgroup \cite{Hilden}; this is the fixed-level counterpart of Birman's stable equivalence theorem for plats \cite{BirmanStable}.

On the braid side, $\B_{2n}$ carries the Dehornoy order $<_D$, a left-invariant strict total order \cite{Dehornoy,FGRRW}. It does not descend to $\mathcal P_{2n}$, and a double coset has no $<_D$--least element: since $\sigma_1\in H_{2n}$ and $1<_D\sigma_1$, left-invariance gives
\[
\cdots<_D\beta\sigma_1^{-2}<_D\beta\sigma_1^{-1}<_D\beta
\]
for every $\beta\in\B_{2n}$, and all these braids lie in $[\beta]_H$. This is why we first minimize a proper complexity function. The purpose of this paper is to introduce, through Hovland's correspondence, Dehornoy-type orderings on bridge positions, and to use them to single out a canonical bridge position of each link type at each bridge level.

Let
\[
c_n\colon \B_{2n}\to \N
\]
be a proper complexity function, that is, one with finite sublevel sets (for instance the Artin word length). For $C\in \mathcal P_{2n}$, let $c_n(C)$ be the minimal complexity of an element of $C$, and let $r_{D,n}(C)$ be the $<_D$--least braid among the finitely many elements of $C$ of complexity $c_n(C)$.

\medskip
\noindent\textbf{Theorem~\ref{thm:order-on-plat-classes}.}
For each $n\ge1$, the relations
\[
C_1\prec_{D,n} C_2 \iff r_{D,n}(C_1)<_D r_{D,n}(C_2)
\]
and
\[
C_1\mathrel{\lhd_{D,n}} C_2 \iff \bigl(c_n(C_1),r_{D,n}(C_1)\bigr)<_{\mathrm{lex}}\bigl(c_n(C_2),r_{D,n}(C_2)\bigr)
\]
are strict total orders on $\mathcal P_{2n}$, and hence on the set of bridge isotopy classes of $n$--bridge positions. Moreover, $\mathrel{\lhd_{D,n}}$ is a well-order, of order type $\omega$ for $n\ge2$.

\medskip
For a link type $\mathcal L$ with bridge number $b(\mathcal L)\le n$, let $\mathcal P^{(n)}(\mathcal L)\subset\mathcal P_{2n}$ be the set of bridge isotopy classes of $n$--bridge positions of $\mathcal L$. Let $\beta^{(n)}_{D,\mathcal L}$ be the $<_D$--least braid among the braids of minimal complexity whose plat closure represents $\mathcal L$, and put
\[
\CanPlat_D^{(n)}(\mathcal L):=[\beta^{(n)}_{D,\mathcal L}]_H .
\]

\medskip
\noindent\textbf{Theorem~\ref{thm:fixed-level-canonical}.}
Let $n\ge b(\mathcal L)$. Then:
\begin{enumerate}
\item $\CanPlat_D^{(n)}(\mathcal L)$ is the $\mathrel{\lhd_{D,n}}$--least element of $\mathcal P^{(n)}(\mathcal L)$;
\item it is the $\prec_{D,n}$--least element of the finite subset of $\mathcal P^{(n)}(\mathcal L)$ consisting of the classes of minimal complexity;
\item the globally defined braid $\beta^{(n)}_{D,\mathcal L}$ coincides with the cosetwise canonical representative $r_{D,n}(\CanPlat_D^{(n)}(\mathcal L))$.
\end{enumerate}

\medskip
The restriction to classes of minimal complexity in (2) cannot be dropped: Example~\ref{ex:not-least} shows that, for a suitable proper complexity function, $\CanPlat_D^{(n)}(\mathcal L)$ is not $\prec_{D,n}$--least in $\mathcal P^{(n)}(\mathcal L)$. Thus the complexity-refined order $\mathrel{\lhd_{D,n}}$ is the order that singles out the canonical bridge position, while the Dehornoy order governs the comparison within each complexity level.

Finally, we consider the following fixed-level finiteness problem: does every link type admit only finitely many bridge isotopy classes of $n$--bridge positions for each $n$?

\medskip
\noindent\textbf{Theorem~\ref{thm:fixed-level-finiteness-equivalence}.}
Let $n\ge b(\mathcal L)$. The following are equivalent:
\begin{enumerate}
\item $\mathcal P^{(n)}(\mathcal L)$ is finite;
\item there is $N=N(n,\mathcal L)$ with $c_n(r_{D,n}(C))\le N$ for every $C\in \mathcal P^{(n)}(\mathcal L)$;
\item $\mathcal P^{(n)}(\mathcal L)$ has a $\mathrel{\lhd_{D,n}}$--greatest element.
\end{enumerate}

\medskip
In a related direction, Hovland \cite{Hovland2} has recently shown that every Hilden double coset contains a positive braid representative whose Dehornoy minimum is at most $p(2n-1)$, where $p$ is the number of permutation braids in the Garside decomposition of a given representative.

It is important to distinguish bridge positions from bridge decompositions. In this paper, bridge positions are considered up to bridge isotopy, whereas bridge decompositions and bridge spheres are finer objects; a single bridge position may admit infinitely many non-isotopic bridge decompositions, for example by twisting along an essential torus \cite{Jang2010}. For this reason, the finiteness problem is formulated for bridge positions up to bridge isotopy.

The paper is organized as follows. Section~\ref{sec:prelim} recalls bridge positions, bridge decompositions, and Hovland's correspondence. Section~\ref{sec:orders} introduces the two orders and proves Theorem~\ref{thm:order-on-plat-classes}. Section~\ref{sec:link-type} proves Theorem~\ref{thm:fixed-level-canonical} and gives Example~\ref{ex:not-least}. Section~\ref{sec:finiteness} proves Theorem~\ref{thm:fixed-level-finiteness-equivalence}. Section~\ref{sec:examples} treats the minimal bridge level and examples, and Section~\ref{sec:questions} lists further questions.

\section{Bridge positions and plat presentation classes}\label{sec:prelim}

Throughout the paper, the ambient space is $S^3$. We fix a standard height function
$h\colon S^3\to \R$
with exactly two critical points.

\subsection{Bridge positions and bridge isotopy}

\begin{definition}
Let $\mathcal L$ be a link type in $S^3$, and let $n$ be a positive integer. An \emph{$n$--bridge position} of $\mathcal L$ is a link $L\in \mathcal L$ such that:
\begin{itemize}
    \item the function $h|_L$ has exactly $2n$ critical points,
    \item all these critical points are non-degenerate, and
    \item every local maximum value of $h|_L$ is greater than every local minimum value of $h|_L$.
\end{itemize}
By a \emph{bridge position} of $\mathcal L$ we mean an $m$--bridge position of $\mathcal L$ for some positive integer $m$.
\end{definition}

\begin{definition}
Let $L_0$ and $L_1$ be bridge positions of the same link type $\mathcal L$. We say that $L_0$ and $L_1$ are \emph{bridge isotopic} if there exists an ambient isotopy
$\{H_t\colon S^3\to S^3\}_{t\in[0,1]}$
such that $H_0=\mathrm{id}$, $H_1(L_0)=L_1$, and $H_t(L_0)$ is a bridge position of $\mathcal L$ for every $t\in[0,1]$.
\end{definition}

\subsection{Bridge decompositions and bridge spheres}

We record the bridge decomposition viewpoint separately, since it is distinct from bridge position and will be used only for comparison.

\begin{definition}
Let $\mathcal L$ be a link type in $S^3$, let $L\in \mathcal L$, and let $n$ be a positive integer. An \emph{$n$--bridge decomposition} of $L$ is a pair $(B^-,B^+)$ of $3$--balls such that:
\begin{itemize}
    \item $B^-\cup B^+=S^3$ and $B^-\cap B^+=\partial B^-=\partial B^+$,
    \item the $2$--sphere $P=B^-\cap B^+$ intersects $L$ transversely, and
    \item for each $\varepsilon\in\{-,+\}$, the tangle $L\cap B^{\varepsilon}$ consists of $n$ arcs simultaneously parallel to $\partial B^{\varepsilon}$.
\end{itemize}
The sphere $P$ is called an \emph{$n$--bridge sphere} of $L$.
\end{definition}

\begin{definition}
Two bridge decompositions $(B^-_0,B^+_0)$ and $(B^-_1,B^+_1)$ of the same link $L$ are \emph{isotopic as bridge decompositions} if there exists an ambient isotopy
$\{H_t\}_{t\in[0,1]}$
such that $H_0=\mathrm{id}$, $H_t(L)=L$, $H_1(B^-_0)=B^-_1$, and $(H_t(B^-_0),H_t(B^+_0))$ is a bridge decomposition of $L$ for every $t\in[0,1]$. Two bridge spheres of $L$ are isotopic as bridge spheres if the analogous condition holds for the spheres.
\end{definition}

\begin{remark}\label{rem:position-vs-decomposition}
Bridge decompositions, bridge spheres, and bridge positions should not be conflated. A single bridge position may admit infinitely many non-isotopic bridge decompositions, for example by twisting along an essential torus; see Jang \cite{Jang2010}. Thus the fixed-level finiteness problem considered in this paper concerns bridge positions up to bridge isotopy, rather than bridge decompositions or bridge spheres.
\end{remark}

\subsection{Plat presentations and Hilden double cosets}

Let $\B_{2n}$ be the braid group on $2n$ strands. For $\beta\in \B_{2n}$, let $\pl(\beta)$ denote the plat closure of $\beta$, obtained by joining the top endpoints and the bottom endpoints in adjacent pairs by caps and cups (Figure~\ref{fig:plat}(a)); placed with respect to $h$ in the standard way, $\pl(\beta)$ is an $n$--bridge position. Let $\tau_n\subset B^3$ denote the standard trivial $n$--string tangle determined by these adjacent pairings. The \emph{Hilden subgroup} $H_{2n}\le \B_{2n}$ is the subgroup of braids whose boundary action on the $2n$ marked points extends to a homeomorphism of the pair $(B^3,\tau_n)$ \cite{Hilden}; see Figure~\ref{fig:plat}(b). We write
\[
\mathcal P_{2n}:=H_{2n}\backslash \B_{2n}/H_{2n},
\qquad
[\beta]_H:=H_{2n}\beta H_{2n}\in \mathcal P_{2n}.
\]

\begin{figure}[ht]
\centering
\begin{tikzpicture}[scale=0.62, line width=0.8pt,
  strand/.style={line width=0.9pt}]
  \begin{scope}
    \draw[-{Stealth}, thin] (-1.3,-1.4) -- (-1.3,3.6) node[above] {$h$};
    \foreach \x in {0,2,4} {
      \draw[strand] (\x,2.0) -- (\x,2.4) arc (180:0:0.5) -- (\x+1,2.0);
      \draw[strand] (\x,0.0) -- (\x,-0.4) arc (180:360:0.5) -- (\x+1,0.0);
    }
    \draw[fill=white] (-0.3,0) rectangle (5.3,2.0);
    \node at (2.5,1.0) {$\beta$};
    \draw[dashed, thin] (-0.9,1.0) -- (-0.35,1.0);
    \draw[dashed, thin] (5.35,1.0) -- (6.0,1.0) node[right] {$P=h^{-1}(v)$};
    \node[font=\small] at (2.5,3.55) {$n$ maxima};
    \node[font=\small] at (2.5,-1.55) {$n$ minima};
    \node at (2.5,-2.5) {(a)};
  \end{scope}
  \begin{scope}[xshift=11.5cm]
    \foreach \x in {0,2,4} {
      \draw[strand] (\x,2.9) -- (\x,3.3) arc (180:0:0.5) -- (\x+1,2.9);
      \draw[strand] (\x,-0.9) -- (\x,-1.3) arc (180:360:0.5) -- (\x+1,-0.9);
      \draw[strand] (\x,2.0) -- (\x,2.2);
      \draw[strand] (\x+1,2.0) -- (\x+1,2.2);
      \draw[strand] (\x,0.0) -- (\x,-0.2);
      \draw[strand] (\x+1,0.0) -- (\x+1,-0.2);
    }
    \draw[fill=black!8] (-0.3,2.2) rectangle (5.3,2.9);
    \node at (2.5,2.55) {$h_1$};
    \draw[fill=white] (-0.3,0) rectangle (5.3,2.0);
    \node at (2.5,1.0) {$\beta$};
    \draw[fill=black!8] (-0.3,-0.9) rectangle (5.3,-0.2);
    \node at (2.5,-0.55) {$h_2$};
    \draw[decorate,decoration={brace,amplitude=5pt}]
      (5.6,2.9) -- (5.6,-0.9) node[midway,right=7pt] {$\beta'=h_1\beta h_2$};
    \node at (2.5,-2.5) {(b)};
  \end{scope}
\end{tikzpicture}
\caption{(a) The plat closure $\pl(\beta)$ of $\beta\in\B_{2n}$ ($n=3$) is an $n$--bridge position: all $n$ maxima lie on the caps, all $n$ minima on the cups, and any level sphere $P$ through the braid region is a bridge sphere. (b) For $h_1,h_2\in H_{2n}$ (braids composed from top to bottom), the inserted braids extend over the trivial tangles of caps and of cups, so $\pl(h_1\beta h_2)$ is bridge isotopic to $\pl(\beta)$. Hovland's theorem (Theorem~\ref{thm:plat-bridge-bijection}) says that this is the only ambiguity.}
\label{fig:plat}
\end{figure}

\begin{theorem}[Hovland {\cite{Hovland}}]\label{thm:plat-bridge-bijection}
For each integer $n\ge 1$, the map $[\beta]_H\mapsto$ (bridge isotopy class of $\pl(\beta)$) is a well-defined bijection from $\mathcal P_{2n}$ to the set of bridge isotopy classes of $n$--bridge positions in $S^3$.
\end{theorem}

This is the fixed-level counterpart of Birman's theorem \cite{BirmanStable}, which describes plat presentations of the same link up to Hilden double cosets and stabilization. From now on we identify $\mathcal P_{2n}$ with the set of bridge isotopy classes of $n$--bridge positions, so that every order on $\mathcal P_{2n}$ is an order on $n$--bridge positions.

\section{Two Dehornoy-type orders at a fixed level}\label{sec:orders}

Fix an integer $n\ge 1$. Let $<_D$ denote the Dehornoy order on $\B_{2n}$; see Dehornoy \cite{Dehornoy} and Fenn--Greene--Rolfsen--Rourke--Wiest \cite{FGRRW}.

\begin{definition}
A function
\[
c_n\colon \B_{2n}\to \N
\]
is called a \emph{proper complexity function at level $n$} if, for every $N\in\N$, the set
\[
\{\beta\in \B_{2n}\mid c_n(\beta)\le N\}
\]
is finite.
\end{definition}

\begin{remark}
The word length with respect to any finite generating set of $\B_{2n}$ is a proper complexity function; examples are the Artin word length and the word length with respect to the simple braids (the divisors of the Garside element $\Delta$). By contrast, the Garside canonical length $\sup-\inf$ is not proper, since it vanishes on every power of $\Delta$.
\end{remark}

Fix such a proper complexity function $c_n$.

\begin{definition}
For $C\in \mathcal P_{2n}$, define
\[
c_n(C):=\min\{c_n(\beta)\mid \beta\in C\},
\qquad
M_n(C):=\{\beta\in C\mid c_n(\beta)=c_n(C)\}.
\]
Since $c_n$ is proper, $M_n(C)$ is finite and nonempty. We define the \emph{Dehornoy canonical representative} of $C$ by
\[
r_{D,n}(C):=\minD\, M_n(C).
\]
\end{definition}

Note that $c_n(r_{D,n}(C))=c_n(C)$, and that $r_{D,n}(C_1)\ne r_{D,n}(C_2)$ whenever $C_1\ne C_2$, since Hilden double cosets partition $\B_{2n}$.

\begin{definition}
For $C_1,C_2\in \mathcal P_{2n}$, define
\[
C_1\prec_{D,n} C_2
\quad\Longleftrightarrow\quad
r_{D,n}(C_1)<_D r_{D,n}(C_2),
\]
and
\[
C_1\mathrel{\lhd_{D,n}} C_2
\quad\Longleftrightarrow\quad
\begin{cases}
c_n(C_1)<c_n(C_2), \quad\text{or}\\
c_n(C_1)=c_n(C_2)\ \text{and}\ r_{D,n}(C_1)<_D r_{D,n}(C_2).
\end{cases}
\]
We call $\prec_{D,n}$ the \emph{Dehornoy-induced order} and $\mathrel{\lhd_{D,n}}$ the \emph{complexity-refined Dehornoy order} on $n$--bridge positions.
\end{definition}

\begin{theorem}\label{thm:order-on-plat-classes}
For each integer $n\ge 1$:
\begin{enumerate}
\item $\prec_{D,n}$ and $\mathrel{\lhd_{D,n}}$ are strict total orders on $\mathcal P_{2n}$;
\item the two orders agree on each complexity level $\{C\in\mathcal P_{2n}\mid c_n(C)=k\}$, and each such level is finite;
\item $\mathrel{\lhd_{D,n}}$ is a well-order on $\mathcal P_{2n}$; every element has only finitely many $\mathrel{\lhd_{D,n}}$--predecessors. In particular, $\mathcal P_{2n}$ has order type $\omega$ with respect to $\mathrel{\lhd_{D,n}}$ for $n\ge 2$, while $\mathcal P_2$ is a single point.
\end{enumerate}
\end{theorem}

\begin{proof}
(1) The maps $C\mapsto r_{D,n}(C)\in\B_{2n}$ and $C\mapsto (c_n(C),r_{D,n}(C))\in \N\times \B_{2n}$ are injective, and $<_D$ and the lexicographic order on $\N\times\B_{2n}$ (with the usual order on $\N$ and $<_D$ on $\B_{2n}$) are strict total orders. A strict total order pulled back along an injective map is a strict total order.

(2) On a single level the first coordinates coincide, so $\mathrel{\lhd_{D,n}}$ reduces to $\prec_{D,n}$. The map $C\mapsto r_{D,n}(C)$ sends the level $\{c_n(C)=k\}$ injectively into the finite set $\{\beta\in\B_{2n}\mid c_n(\beta)\le k\}$.

(3) If $C'\mathrel{\lhd_{D,n}}C$, then $c_n(C')\le c_n(C)$, so $C'$ lies in one of the finitely many finite levels $k\le c_n(C)$. Hence every element has finitely many predecessors, which implies that $\mathrel{\lhd_{D,n}}$ is a well-order and that an infinite $\mathcal P_{2n}$ has order type $\omega$. For $n=1$, the group $\B_2$ is generated by $\sigma_1$, which lies in $H_2$, so $\mathcal P_2$ is a single point. For $n\ge 2$, the set $\mathcal P_{2n}$ is infinite, since there are infinitely many distinct $2$--bridge knot types and each has an $n$--bridge position for every $n\ge 2$ (by stabilization).
\end{proof}

\begin{remark}\label{rem:role-of-dehornoy}
The well-ordering in Theorem~\ref{thm:order-on-plat-classes}(3) comes from the complexity function; within a complexity level, the comparison is made by the Dehornoy order. In contrast, $\prec_{D,n}$ compares classes of different complexity purely through $<_D$, which is not a well-order on $\B_{2n}$. The difference between the two orders, illustrated in Figure~\ref{fig:orders}, is precisely what is responsible for Example~\ref{ex:not-least} below.
\end{remark}

\begin{figure}[ht]
\centering
\begin{tikzpicture}[scale=0.85,
  open/.style={circle, draw, fill=white, inner sep=1.6pt},
  full/.style={circle, draw, fill=black, inner sep=1.6pt}]
  \draw[-{Stealth}] (0,0) -- (6.4,0) node[right] {$c_n(C)$};
  \draw[-{Stealth}] (0,0) -- (0,5.2) node[above] {$r_{D,n}(C)$ under $<_D$};
  \foreach \k/\lab in {1/1,2/2,3/3,4/4} {
    \draw (1.3*\k,0.08) -- (1.3*\k,-0.08) node[below] {$\lab$};
    \draw[gray!60, thin] (1.3*\k,0.25) -- (1.3*\k,4.9);
  }
  \node[open] at (1.3,1.8) {};
  \node[full] (C2) at (1.3,3.6) {};
  \node[full] (C1) at (2.6,1.0) {};
  \node[open] at (2.6,2.6) {};
  \node[open] at (2.6,4.3) {};
  \node[open] at (3.9,0.6) {};
  \node[full] at (3.9,2.2) {};
  \node[open] at (3.9,3.1) {};
  \node[open] at (3.9,4.6) {};
  \node[open] at (5.2,1.5) {};
  \node[open] at (5.2,2.9) {};
  \node[full] at (5.2,3.9) {};
  \node at (5.85,2.6) {$\cdots$};
  \node[left=2pt, fill=white, inner sep=1pt] at (C2) {$C_2$};
  \node[left=2pt, fill=white, inner sep=1pt] at (C1) {$C_1$};
\end{tikzpicture}
\caption{A schematic picture of the two orders on $\mathcal P_{2n}$. Each class $C$ is placed in the column $c_n(C)$ at the height of $r_{D,n}(C)$ with respect to $<_D$; each column is finite. The order $\prec_{D,n}$ compares heights only, while $\mathrel{\lhd_{D,n}}$ reads the columns from left to right and each column from bottom to top. Filled dots indicate the classes in $\mathcal P^{(n)}(\mathcal L)$. In the situation of Example~\ref{ex:not-least}, $\CanPlat_D^{(n)}(\mathcal L)=C_2$ is $\mathrel{\lhd_{D,n}}$--least among the filled dots, whereas $C_1\prec_{D,n}C_2$.}
\label{fig:orders}
\end{figure}

\section{Canonical bridge positions of a link type}\label{sec:link-type}

Let $\mathcal L$ be a link type in $S^3$, and let $b(\mathcal L)$ denote its bridge number.

\begin{definition}
For each integer $n\ge b(\mathcal L)$, define
\[
\mathcal P^{(n)}(\mathcal L):=
\{C\in \mathcal P_{2n}\mid \pl(\beta)\in \mathcal L \text{ for some } \beta\in C\}.
\]
For $n<b(\mathcal L)$, we set $\mathcal P^{(n)}(\mathcal L):=\varnothing$.
\end{definition}

By Theorem~\ref{thm:plat-bridge-bijection}, if $\pl(\beta)\in\mathcal L$ for some $\beta\in C$, then $\pl(\beta')\in\mathcal L$ for every $\beta'\in C$, and $\mathcal P^{(n)}(\mathcal L)$ is identified with the set of bridge isotopy classes of $n$--bridge positions of $\mathcal L$. It is nonempty for $n\ge b(\mathcal L)$.

Fix $n\ge b(\mathcal L)$.

\begin{definition}
Let
\[
\mathcal B^{(n)}(\mathcal L):=\{\beta\in \B_{2n}\mid \pl(\beta)\in \mathcal L\},
\qquad
c_{\min}^{(n)}(\mathcal L):=\min\{c_n(\beta)\mid \beta\in \mathcal B^{(n)}(\mathcal L)\},
\]
\[
M_{\mathcal L}^{(n)}:=\{\beta\in \mathcal B^{(n)}(\mathcal L)\mid c_n(\beta)=c_{\min}^{(n)}(\mathcal L)\}.
\]
Since $c_n$ is proper, $M_{\mathcal L}^{(n)}$ is finite and nonempty. Define
\[
\beta_{D,\mathcal L}^{(n)}:=\minD\, M_{\mathcal L}^{(n)},
\qquad
\CanPlat_D^{(n)}(\mathcal L):=[\beta_{D,\mathcal L}^{(n)}]_H\in \mathcal P^{(n)}(\mathcal L).
\]
We call $\CanPlat_D^{(n)}(\mathcal L)$ the \emph{canonical $n$--bridge position} of $\mathcal L$ (with respect to $c_n$). We also put
\[
\mathcal P^{(n)}_{\mathrm{mc}}(\mathcal L):=\{C\in \mathcal P^{(n)}(\mathcal L)\mid c_n(C)=c^{(n)}_{\min}(\mathcal L)\},
\]
the set of $n$--bridge positions of $\mathcal L$ of minimal complexity.
\end{definition}

Since $\mathcal B^{(n)}(\mathcal L)=\bigcup_{C\in\mathcal P^{(n)}(\mathcal L)}C$, we have
\begin{equation}\label{eq:cmin}
c^{(n)}_{\min}(\mathcal L)=\min\{c_n(C)\mid C\in \mathcal P^{(n)}(\mathcal L)\},
\end{equation}
and for $C\in \mathcal P^{(n)}(\mathcal L)$,
\begin{equation}\label{eq:Mn-in-ML}
c_n(C)=c^{(n)}_{\min}(\mathcal L)
\ \Longrightarrow\
M_n(C)=M^{(n)}_{\mathcal L}\cap C\subset M^{(n)}_{\mathcal L}.
\end{equation}
If instead $c_n(C)>c^{(n)}_{\min}(\mathcal L)$, then $M_n(C)\cap M^{(n)}_{\mathcal L}=\varnothing$.

\begin{proposition}\label{prop:global-vs-cosetwise}
We have
\[
\beta_{D,\mathcal L}^{(n)}=r_{D,n}\bigl(\CanPlat_D^{(n)}(\mathcal L)\bigr).
\]
\end{proposition}

\begin{proof}
Let $\beta^*:=\beta_{D,\mathcal L}^{(n)}$ and $C^*:=[\beta^*]_H$. Since $\beta^*\in C^*$ and $C^*\in\mathcal P^{(n)}(\mathcal L)$, \eqref{eq:cmin} gives
\[
c^{(n)}_{\min}(\mathcal L)\le c_n(C^*)\le c_n(\beta^*)=c_{\min}^{(n)}(\mathcal L),
\]
so $c_n(C^*)=c^{(n)}_{\min}(\mathcal L)$. By \eqref{eq:Mn-in-ML}, $M_n(C^*)=M_{\mathcal L}^{(n)}\cap C^*$. This set contains $\beta^*$, and $\beta^*$ is the $<_D$--least element of the larger set $M^{(n)}_{\mathcal L}$. Hence $r_{D,n}(C^*)=\minD M_n(C^*)=\beta^*$.
\end{proof}

\begin{theorem}\label{thm:fixed-level-canonical}
Let $\mathcal L$ be a link type, let $n\ge b(\mathcal L)$, and let $C^*:=\CanPlat_D^{(n)}(\mathcal L)$. Then:
\begin{enumerate}
\item $C^*$ is the $\mathrel{\lhd_{D,n}}$--least element of $\mathcal P^{(n)}(\mathcal L)$;
\item the set $\mathcal P^{(n)}_{\mathrm{mc}}(\mathcal L)$ is finite, contains $C^*$, and $C^*$ is its $\prec_{D,n}$--least element;
\item $\beta_{D,\mathcal L}^{(n)}=r_{D,n}(C^*)$.
\end{enumerate}
\end{theorem}

\begin{proof}
Write $\beta^*:=\beta^{(n)}_{D,\mathcal L}$. Part (3) is Proposition~\ref{prop:global-vs-cosetwise}, and its proof shows $C^*\in\mathcal P^{(n)}_{\mathrm{mc}}(\mathcal L)$.

(2) The map $C\mapsto r_{D,n}(C)$ is injective and, by \eqref{eq:Mn-in-ML}, sends $\mathcal P^{(n)}_{\mathrm{mc}}(\mathcal L)$ into the finite set $M^{(n)}_{\mathcal L}$; hence $\mathcal P^{(n)}_{\mathrm{mc}}(\mathcal L)$ is finite. Let $C\in \mathcal P^{(n)}_{\mathrm{mc}}(\mathcal L)$ with $C\ne C^*$. Then $r_{D,n}(C)\in M^{(n)}_{\mathcal L}$ and $r_{D,n}(C)\ne r_{D,n}(C^*)=\beta^*$, so $\beta^*<_D r_{D,n}(C)$ by the minimality of $\beta^*$ in $M^{(n)}_{\mathcal L}$. That is, $C^*\prec_{D,n}C$.

(1) Let $C\in \mathcal P^{(n)}(\mathcal L)$ with $C\ne C^*$. By \eqref{eq:cmin}, $c_n(C)\ge c^{(n)}_{\min}(\mathcal L)=c_n(C^*)$. If the inequality is strict, then $C^*\mathrel{\lhd_{D,n}}C$ by definition. If equality holds, then $C\in \mathcal P^{(n)}_{\mathrm{mc}}(\mathcal L)$, and (2) gives $r_{D,n}(C^*)<_D r_{D,n}(C)$, hence again $C^*\mathrel{\lhd_{D,n}}C$.
\end{proof}

The following example shows that, in Theorem~\ref{thm:fixed-level-canonical}(2), the restriction to classes of minimal complexity cannot be removed, that is, $C^*$ need not be least in $\mathcal P^{(n)}(\mathcal L)$ for the Dehornoy-induced order $\prec_{D,n}$.

\begin{example}\label{ex:not-least}
Let $\mathcal L$ and $n\ge b(\mathcal L)$ be such that $\mathcal P^{(n)}(\mathcal L)$ contains two distinct classes $C_1\ne C_2$; for instance, by Montesinos \cite{Montesinos} there are knots with two minimal plat presentations that are not Hilden double coset equivalent, so one may take $n=b(\mathcal L)$ for such a knot. Choose $\alpha\in C_1$ and $\beta\in C_2$; after interchanging the indices if necessary, we may assume $\alpha<_D\beta$. Define
\[
c_n(\beta)=1,\qquad c_n(\alpha)=2,\qquad c_n(\gamma)=3+\ell(\gamma)\quad(\gamma\ne\alpha,\beta),
\]
where $\ell$ denotes the Artin word length. This is a proper complexity function. Then $c^{(n)}_{\min}(\mathcal L)=1$, $M^{(n)}_{\mathcal L}=\{\beta\}$, and $\CanPlat_D^{(n)}(\mathcal L)=C_2$. On the other hand, $\beta\notin C_1$, so $c_n(C_1)=2$, $M_n(C_1)=\{\alpha\}$ and $r_{D,n}(C_1)=\alpha$. Therefore
\[
r_{D,n}(C_1)=\alpha<_D\beta=r_{D,n}(C_2),
\]
that is, $C_1\prec_{D,n}\CanPlat_D^{(n)}(\mathcal L)$, although $\CanPlat_D^{(n)}(\mathcal L)\mathrel{\lhd_{D,n}}C_1$; see Figure~\ref{fig:orders}.
\end{example}

\begin{remark}
If $\mathcal P^{(n)}(\mathcal L)$ is finite, it has a $\prec_{D,n}$--least element, but by Example~\ref{ex:not-least} this element need not be $\CanPlat_D^{(n)}(\mathcal L)$. If $\mathcal P^{(n)}(\mathcal L)$ is infinite, we do not know whether a $\prec_{D,n}$--least element exists at all; see Question~\ref{q:prec-least}.
\end{remark}

\section{A fixed-level bridge finiteness conjecture}\label{sec:finiteness}

\begin{conjecture}[Fixed-level bridge finiteness conjecture]\label{conj:fixed-bridge-finiteness}
For each integer $n\ge 1$, every link type in $S^3$ admits at most finitely many bridge isotopy classes of $n$--bridge positions.
\end{conjecture}

\begin{remark}
By Theorem~\ref{thm:plat-bridge-bijection}, Conjecture~\ref{conj:fixed-bridge-finiteness} is equivalent to the assertion that $\mathcal P^{(n)}(\mathcal L)$ is finite for every link type $\mathcal L$ and every $n\ge 1$. In view of Remark~\ref{rem:position-vs-decomposition}, this formulation is deliberately about bridge positions up to bridge isotopy, not about bridge decompositions or bridge spheres.
\end{remark}

\begin{theorem}\label{thm:fixed-level-finiteness-equivalence}
Let $\mathcal L$ be a link type and let $n\ge b(\mathcal L)$. Then the following are equivalent:
\begin{enumerate}
    \item the set $\mathcal P^{(n)}(\mathcal L)$ is finite;
    \item there exists a constant $N=N(n,\mathcal L)$ such that
    \[
    c_n(r_{D,n}(C))\le N
    \qquad\text{for every }C\in \mathcal P^{(n)}(\mathcal L);
    \]
    \item $\mathcal P^{(n)}(\mathcal L)$ has a $\mathrel{\lhd_{D,n}}$--greatest element.
\end{enumerate}
\end{theorem}

\begin{proof}
(1)$\Rightarrow$(2): take $N$ to be the maximum of the finite set $\{c_n(r_{D,n}(C))\mid C\in \mathcal P^{(n)}(\mathcal L)\}$.

(2)$\Rightarrow$(1): the injective map $C\mapsto r_{D,n}(C)$ sends $\mathcal P^{(n)}(\mathcal L)$ into the finite set $\{\beta\in \B_{2n}\mid c_n(\beta)\le N\}$.

(1)$\Rightarrow$(3): $\mathcal P^{(n)}(\mathcal L)$ is a finite nonempty subset of a totally ordered set.

(3)$\Rightarrow$(1): if $G$ is the greatest element, then every element of $\mathcal P^{(n)}(\mathcal L)$ is $G$ or a $\mathrel{\lhd_{D,n}}$--predecessor of $G$, and $G$ has only finitely many predecessors by Theorem~\ref{thm:order-on-plat-classes}(3).
\end{proof}

\begin{remark}
Since $c_n(r_{D,n}(C))=c_n(C)$, condition (2) is equivalent to boundedness of the minimal complexities $c_n(C)$ over $C\in\mathcal P^{(n)}(\mathcal L)$. The point of (2) is that a bound on a single, canonically chosen braid in each class suffices. Hovland's positive bounded representatives \cite{Hovland2} give another explicit family of representatives in which such bounds may be sought.
\end{remark}

\section{Minimal level and examples}\label{sec:examples}

The minimal-level theory is obtained by setting $n=b(\mathcal L)$.

\begin{definition}
For a knot or link type $\mathcal L$, we write
\[
\mathcal P^{\min}(\mathcal L):=\mathcal P^{(b(\mathcal L))}(\mathcal L)
\quad\text{and}\quad
\CanPlat_D^{\min}(\mathcal L):=\CanPlat_D^{(b(\mathcal L))}(\mathcal L).
\]
\end{definition}

\subsection{The unknot}

\begin{corollary}\label{cor:unknot-fixed}
For the unknot $U$ and every integer $n\ge 1$, the set $\mathcal P^{(n)}(U)$ consists of a single element. Consequently, $\CanPlat_D^{(n)}(U)$ is independent of the chosen proper complexity function $c_n$, and is simultaneously least for $\prec_{D,n}$ and $\mathrel{\lhd_{D,n}}$.
\end{corollary}

\begin{proof}
Every non-minimal bridge position of the unknot destabilizes to the standard $1$--bridge position; see Otal \cite{OtalUnknot} and the survey \cite[Section~5.2.2]{OzawaSurvey}. Hence the unknot has a unique $n$--bridge position up to bridge isotopy for every $n\ge 1$, and $\mathcal P^{(n)}(U)$ is a singleton by Theorem~\ref{thm:plat-bridge-bijection}. Hovland \cite{Hovland} also derives this uniqueness of the Hilden double coset class of the $n$--bridge unknot.
\end{proof}

\subsection{Rational knots}

\begin{corollary}\label{cor:rational-min}
If $K$ is a rational knot, then $\mathcal P^{\min}(K)$ has at most two elements.
\end{corollary}

\begin{proof}
Since $K$ is a rational knot, $b(K)=2$. By Hovland \cite[Theorem~5.5]{Hovland}, which rests on the uniqueness of $2$--bridge spheres of $2$--bridge knots (Schubert \cite{Schubert}), a $2$--bridge knot admits at most two bridge isotopy classes of $2$--bridge positions; the two are related by a reflection exchanging the upper and lower bridges. Hence $|\mathcal P^{\min}(K)|\le 2$.
\end{proof}

For non-minimal levels, Otal \cite{OtalRational} proved that every $n$--bridge presentation of a rational knot with $n\ge 3$ is stabilized.

\subsection{Torus knots}

\begin{remark}\label{rem:torus}
For torus knots, the author proved that non-minimal bridge decompositions are stabilized and that $n$--bridge decompositions are unique for every $n$ \cite{OzawaTorus}. Combining this with the correspondence of Theorem~\ref{thm:plat-bridge-bijection}, Hovland \cite[Theorem~5.4]{Hovland} obtains a single Hilden double coset class for torus knots. Thus, for a torus knot $K$ and each fixed $n\ge b(K)$, the set $\mathcal P^{(n)}(K)$ is a singleton, and $\CanPlat_D^{(n)}(K)$ is independent of $c_n$, as in Corollary~\ref{cor:unknot-fixed}.
\end{remark}

\section{Further questions}\label{sec:questions}

\begin{question}
For a fixed level $n$, to what extent does $\CanPlat_D^{(n)}(\mathcal L)$ depend on the choice of the proper complexity function $c_n$? In particular, for natural choices such as the Artin word length or the word length with respect to the simple braids, can one identify classes of links with $|\mathcal P^{(n)}(\mathcal L)|\ge 2$ for which it is independent of $c_n$?
\end{question}

\begin{question}\label{q:prec-least}
For natural complexity functions such as the Artin word length, is $\CanPlat_D^{(n)}(\mathcal L)$ the $\prec_{D,n}$--least element of $\mathcal P^{(n)}(\mathcal L)$? More generally, does $\mathcal P^{(n)}(\mathcal L)$ always have a $\prec_{D,n}$--least element?
\end{question}

\begin{question}
How does $\CanPlat_D^{(n)}(\mathcal L)$ behave under stabilization from level $n$ to level $n+1$? Is there a natural compatibility between $\CanPlat_D^{(n)}(\mathcal L)$ and $\CanPlat_D^{(n+1)}(\mathcal L)$?
\end{question}

\begin{question}
How should the orders $\prec_{D,n}$ and $\mathrel{\lhd_{D,n}}$ be interpreted in the double branched cover? More precisely, can they and the canonical position $\CanPlat_D^{(n)}(\mathcal L)$ be reformulated in terms of genus $n-1$ hyperelliptic Heegaard splittings of the double branched cover of $\mathcal L$?
\end{question}

\section*{Acknowledgements}
The author used ChatGPT (OpenAI) and Claude (Anthropic) to assist with language refinement, structural editing, and checking of the manuscript. All mathematical content, statements, proofs, and references were checked and approved by the author, who takes full responsibility for the final text.

\end{document}